\documentclass[11pt,leqno,amscd,amssymb,verbatim, url]{amsart}%
\usepackage{amsfonts,amssymb,amsthm}
\usepackage{amsmath,amscd}
\usepackage{amsfonts,latexsym}
\usepackage{wrapfig}
\usepackage{amsmath}
\usepackage{amsfonts}
\usepackage{amssymb}
\usepackage{graphicx}%
\newtheorem*{algorithm}{Algorithm}
\newtheorem*{proposition}{Proposition}

\theoremstyle{remark}
\newtheorem*{remark}{Remark}
\newcommand{\blist}{\begin{list}{\rom{(\roman{enumi})}}{\setlength
{\leftmarg in}{0em} \setlength{\itemindent}{7ex}
\setlength{\labelsep}{2ex}\setlength{\listparindent}{\parindent}
\usecounter{enumi}}}
\newcommand{\elist}{\end{list}}

\begin{document}
\title[Kazhdan--Lusztig Basis Elements ]{{\large \textbf{Computing Individual Kazhdan--Lusztig Basis Elements}}}
\author{Leonard L. Scott}
\address{Department of Mathematics \\
University of Virginia\\
Charlottesville, VA 22903}
\email{lls2l@virginia.edu {\text{\textrm{(Scott)}}}}
\author{Timothy Sprowl}
\address{9170 Ivy Springs Place\\ Mechanicsville\\ VA 23116.}
\email{tim.spr@gmail.com {\text{\textrm{(Sprowl)}}}}
\thanks{Some of the results of this paper were written in the second author's 2013
thesis for the computer science major at the University of Virginia, reporting
on a project with the first author supported by the National Science Foundation.}

\begin{abstract}
In well-known work, Kazhdan and Lusztig (1979) defined a new set of Hecke
algebra basis elements (actually two such sets) associated to elements in any
Coxeter group. Often these basis elements are computed by a standard recursive
algorithm which, for Coxeter group elements of long length, generally involves
computing most basis elements corresponding to Coxeter group elements of
smaller length. Thus, many calculations simply compute all basis elements
associated to a given length or less, even if the interest is in a specific
Kazhdan-Lusztig basis element. Similar remarks apply to ``parabolic" versions
of these basis elements defined later by Deodhar (1987,1990), though the
lengths involved are the (smaller) lengths of distinguished coset
representatives. We give an algorithm which targets any given Kazhdan-Lusztig
basis element or parabolic analog and does not  precompute any other Kazhdan-Lusztig
basis elements. In particular it does not have to store them.  This results in
a considerable saving in memory usage, enabling new calculations in an
important case (for finite and algebraic group 1-cohomology with irreducible
coefficients) analyzed by Scott-Xi (2010).
\end{abstract}
\maketitle

\section{Introduction}

This note addresses a need we have perceived for a non-recursive algorithm
focused on determining coefficients in Kazhdan--Lusztig polynomials $P_{x,y}$
associated to a single $y$ in a given Coxeter group $W$, or equivalently, to
that of a single Kazhdan--Lusztig Hecke algebra basis element $C_{y}^{\prime}$
in the notation of \cite{KL} or \cite[p.~101]{De2}. Our approach here applies
also to the parabolic Kazhdan--Lusztig polynomials $P_{x,y}^{J}$ and basis
elements $^{J}C_{y}^{\prime}$ (for an appropriate Hecke algebra right module
$M= M^{J}$) in the notation of \cite[p.~113]{De2}. The parabolic notations
are defined only for $y$ ``distinguished" (shortest) in its right coset
$W_{J}y$ in $W$, and there is a similar requirement on $x.$

We follow the notation of \cite{De2} closely. The Hecke algebra of
$W$ is denoted $\mathcal{H}$. It is a free $R$-module, where $R$ is
the ring $\mathbb{Z}[q^{1/2},q^{-1/2}]$, with basis elements
$T_{x},$ $x\in W$, as discussed in \cite[\S 3]{De2}, following
standard terminology. The identity element of $W$ is denoted $e,$
and $T_{e}$ is the identity of the ring $\mathcal{H}$. The set $J$
is a subset of the set $S$ of fundamental generators of $W$ and
serves as a set of fundamental generators of the Coxeter group
$W_{J}$. The set of distinguished right coset representatives of
$W_{J}$ in $W$ is denoted $W^{J}.$ Henceforth, we fix a subset $J,$
which may be the empty set. The module $M=M^{J}$ has a basis
$\{m_{x}\}_{x\in W^{J}}$ with $m_{x}=m_{e}T_{x}$ for $x\in W^{J}$
and $m_{e}T_{w}=q^{\ell(w)}m_{e}$ for $w\in W_{J}$. See the
displayed action \cite[p.~113]{De2} of $\mathcal{H}$ on $M$. We
mention that the cited display corrects an earlier misprint in the
middle term of a  similar display \cite[p.~485]{De1}. We also remark
that the modules considered there and here are ``tensor induced"
from evident rank 1 modules for the Hecke algebra corresponding to
$W_J$. (Though $M$ is a right $\mathcal{H}$-module, the action of
the commutative ring $R$ is often written on the left.) With this
terminology, we have
\begin{equation}
^{J}C_{y}^{\prime}=q^{-\ell(y)/2}\sum_{x\leq y}P_{x,y}^{J}(q)m_{e}
T_{x}\text{\ \ \ }(x,y\in W^{J}).  \tag{*}\label{(*)}
\end{equation}
We will return to this equation later. It is part of \cite[Prop.~5.1(i)]{De2},
the parabolic analog of \cite[(1.1.c)]{KL}. If $s\in S,$ we have $^\emptyset C_s^\prime=C_s^\prime=q^{-1/2}(T_e+T_s).$ When the group $W_J$ is finite, with element $w_J^0$ of maximal length, we have
$P_{x,y}^J=P_{w_{J}^0x,w_{J}^0y}.$ See
\cite[Prop.~3.4]{De1}, applied through the duality set-up of \cite[Rem.~2.6]{De3}. It is worth noting that, even when $W_J$ is finite, the basic
recursion \cite[Prop.~5.2(iii)]{De2}\footnote{The reader may notice there is a
misprint in part (ii) of the same proposition \cite[Prop.~5.2]{De2}, where
$-f^{J}$ should simply be $f$, representing the expression $q^{1/2}
+q^{-1/2}$.  This is irrelevant to the recursion in part (iii).} for the
parabolic Kazhdan--Lusztig polynomials $P_{x,y}^{J}$ is much more effective
than the corresponding non-parabolic ($J=\emptyset$) recursion for computing the
polynomials $P_{w_{J}^{0}x,w_{J}^{0}y}.$ We will call \cite[Prop.~5.2(iii)]{De2}  the {\em Deodhar recursion} (to distinguish it from the more elaborate {\em Deodhar algorithm} we will discuss later). Explicitly, the Deodhar recursion states the following, with $^{J}\mu(z,y)$ denoting the
coefficient of $q^{(\ell(y)-\ell(z)-1)/2}$ in $P_{z,y}^{J}$:
\[
\text{Let }y,ys\in W^{J}\text{ with }s\in S\text{ and }y<ys.\text{ Then }%
^{J}C_{y}^{\prime}C_{s}^{\prime}=\text{ }^{J}C_{ys}^{\prime}+
\!\!\!\!\!\! \sum
_{\substack{z\in W^{J}\\ zs<z\mbox{\small{  or }} zs\notin W^{J}}}^{J}  \!\!\!\!\!
^{J}\mu(z,y)C_{z}^{\prime}.
\]
It makes sense also to call the $J=\emptyset$ case, equivalent to \cite[(2.3b)]{KL}
via \cite[(1.1.1c)]{KL}, the {\em Kazhdan--Lusztig  recursion.}

Next, following \cite[p.~114]{De2}, we define, for each finite sequence
$\mathbf{s}=(s_{1},s_{2},\ldots s_{k})$ of elements of $S$ whose product
$\pi(\mathbf{s})=s_{1}s_{2}\cdots s_{k}$ has length $k$, the element%
\begin{equation}
^{J}D_{\mathbf{s}}^{\prime}=m_{e}C_{s_{1}}^{^{\prime}}C_{s_{2}}^{^{\prime}%
}\cdots C_{s_{k}}^{^{\prime}} . \tag{$^JD_\textbf{s}^\prime$}\label{D'}
\end{equation}
In our algorithm we need to compute a lot of these, but, fortunately
for memory requirements, there is no need to store them. In
\cite[Prop.~5.3(i)]{De2} Deodhar gives closed forms for these
elements, though their calculation involves examining subsequences
of $\mathbf{s}$, an operation potentially of exponential time in
$k.$ We have found the simple iterative computation
$m_{e}C_{s_{1}}^{^{\prime}},$ $m_{e}C_{s_{1}}^{^{\prime}}C_{s_{2}}^{^{\prime}%
}, \ldots,m_{e}C_{s_{1}}^{^{\prime}}C_{s_{2}}^{^{\prime}}\cdots C_{s_{k}%
}^{^{\prime}}$  to be a reasonable computational procedure, running in time
at most proportional to $k^{2}|W^{J}(x)|$ in integer operations, where $W^{J}(x)=\{z\in
W^{J}|z<x\}.$ At any iteration, multiplication by a given $C_{s}^{^{\prime}}$
is easily done with the rules \cite[p.~113]{De2} for multiplication on $M$ by
$T_{s}$ mentioned above. Reformulated versions of these rules, in terms of
multiplication by $C_{s}^{\prime}$, are given below.
\[
m_{x}C_{s}^{^{\prime}}=\left\{
\begin{array}{rl}
q^{1/2}(m_{x}+m_{xs}) & \mbox{if } \ell(xs)<\ell(x) , \\[2mm]
q^{-1/2}(m_{x}+m_{xs}) & \mbox{if } \ell(xs)>\ell(x) \mbox{ and } xs\in W^{J} , \\[2mm]
(q^{1/2}+q^{-1/2})m_{x} & \mbox{if } \ell(xs)>\ell(x) \mbox{ and } xs\notin W^{J}.
\end{array}
\right.
\]
Note also from the definition of $^{J}D_{\mathbf{s}}^{\prime}$ that it is
obtained by applying $Z[q]$-linear combinations of elements $T_{x}\,$, $x\in
W$, to $m_{e}/q^{\ell(y)/2}$ and so is a $Z[q]$-linear combinations\ of
elements $m_{x}/q^{\ell(y)/2}$, $\ x\in W^{J}$. Nonzero terms occur only for
$x\leq y,$ and the coefficient of $m_{y}/q^{\ell(y)/2}$ is the element $1\in
R$.  In our algorithm, it will be useful to write elements of $M$ as $R$-linear
combinations of elements $m_{x}/q^{\ell(x)/2}.$ When this is done for
$^{J}D_{\mathbf{s}}^{\prime}$, we find that any power $q^{n/2}$ which
appears with nonzero coefficient in the (Laurent polynomial) coefficient of
$m_{x}/q^{\ell(x)/2}$ satisfies $n\equiv\ell(y)-\ell(x)$ modulo $2$. This
condition is equivalent to the $Z[q]$-coefficient requirements just noted in
the case of elements $m_{x}/q^{\ell(y)/2}$.

Finally, we need the involution $m\mapsto\overline{m}$ on $M$ from \cite[p.~113]{De2}. It satisfies $\overline{rm}=\overline{r}\overline{m}$, where $r\mapsto$ $\overline{r}$ on the ring $R$ sending $q^{1/2}$ to
$q^{-1/2}.$ Also, $\overline{m_{e}}=m_{e}$ and $\overline{mT_{x}}=\overline
{m}\overline{T_{x}}$, where $\overline{T_{x}}=T_{x^{-1}}^{-1}$ ($m\in M$,
$x\in W$). The fixed point space on $M$ of the involution $m\mapsto
\overline{m}$ is denoted $M^{0}.$ Then, according to \cite[Prop.~5.1(i)]{De2},
for each $y\in W^{J}$ there is a unique element $^{J}C_{y}^{\prime}\in M^{0}$
which satisfies equation (*) above for polynomials $P_{x,y}^{J}(q)$ of degree
at most $(\ell(y)-\ell(x)-1)/2$  when $x<y$ and with $P_{y,y}^{J}(q)=1$ ($x\leq y$ elements of $W^{J}$).  We can give a sharper uniqueness result using \cite[Prop.~5.1(ii)]{De2}, which asserts the elements $^{J}C_{y}^{\prime}$ form a basis of $M^{0}$ over the ring $R^{0}$ of invariants of the involution $r\mapsto\overline{r}$ on $R$.

%TCIMACRO{\TeXButton{proposition}{\newtheorem{proposition}{Proposition}
%\begin{proposition}
%Put $t=q^{1/2}$. \ Suppose $y\in W$ $^{J}$, and $^{J}C_{y}^{"}\in M^{0}$ has
%the form $p_{x,y}^{J}(t^{-1})(m_{x}/t^{\ell(x)})$ where  $p_{x,y}^{J}%
%(t^{-1})$ is a
%polynomial in $t^{-1}$ with zero constant term whenever $x<y$, and
%$p_{y,y}^{J}(t^{-1})=1$.Then $^{J}C_{y}^{"}=$ $^{J}C_{y}^{^{\prime}}$ .
%\end{proposition}}}%
%BeginExpansion

\begin{proposition}
Put  $t=q^{1/2}$.  Suppose $y\in W^{J}$, and that $^{J}C_{y}^{''}\in M^{0}$ has
the form
\linebreak
$\sum_{x \leq y}p_{x,y}^{J}(t^{-1})(m_{x}/t^{\ell(x)})$,  where  $p_{x,y}^{J} (t^{-1})$ is a
polynomial in $t^{-1}$ with zero constant term whenever $x<y$, and
$p_{y,y}^{J}(t^{-1})=1$. Then $^{J}C_{y}^{''}=$ $^{J}C_{y}^{^{\prime}}$ .
\end{proposition}

\begin{proof}
Write $^{J}C_{y}^{''}$ as a linear combination of elements $\sum_{z} f_{z}$ $^{J}C_{z}^{^{\prime}}$ with $z\in W^{J}$ and $f_{z}\in R^{0}.$

Comparing coefficients of $m_{z}/t^{\ell(z)},$ we find that any $z$ maximal among those occurring with nonzero $f_{z}$ must be $y,$ and $f_{y}=1.$ Next, suppose some $z<y$ has a nonzero $f_{z}$ and take $z<y$ maximal with
that property. Comparing coefficients of $m_{z}/t^{\ell(z)}$ again, we have
\[
p_{z,y}^{J}(t^{-1})=f_{z}+P_{z,y}^{J}(q)/t^{\ell(y)-\ell(x)}\text{.}%
\]
But both $p_{z,y}^{J}(t^{-1})$ and $P_{z,y}^{J}(q)/t^{\ell(y)-\ell(x)}$ have
nonzero coefficients only for negative powers of $t.$ This property is
inherited by their difference $f_{z}$. However, the element $f_{z}\in R^{0}$
is symmetric with respect to the involution of $R$ interchanging $t$ and
$t^{-1}.$ So it must be that $f_{z}=0\,$, and $^{J}C_{y}^{''}=$ $^{J}%
C_{y}^{^{\prime}}$.
\end{proof}

Continuing with the notation $t=q^{1/2}$, we can now describe our algorithm.
For any element $f(t)$ of $R$, we write $f(t)=f_{\geq0}(t)+f_{<0}(t^{-1})$, where both $f_{\geq0}$, $f_{<0}$ are integer polynomial expressions, and $f_{<0}$ has a zero constant term. Similarly, we let $f_{>0}(t)$ be the
positive degree part of $f_{\geq0}(t).$

%TCIMACRO{\TeXButton{algorithm}{\newtheorem{algorithm}[proposition]{Algorithm}
%\begin{algorithm}
%For any given $y\in W^{J}$ we determine $^{J}C_{y}^{\prime}$ as an $R$-linear
%combination of the basis elements $m_{x}$ of $M$: \ Write $y=s_{1}s_{2}\cdots
%s_{k}$ as a reduced product for a sequence $\mathbf{s}=(s_{1},s_{2}
%,\ldots,s_{k})$ of elements of $S.$ Introduce a temporary variable $Fat$
%$^{J}C_{y}^{\prime}$ , initialized to $^{J}D_{\mathbf{s}}^{\prime}$ and
%written as a linear combination of the elements $m_{x}/t^{\ell(x)}$ , $x\in
%W^{J}$.  \ Next,\ \ we\ look in $Fat$ $^{J}C_{y}^{^{\prime}}$ for any $x<y$
%with a\ Laurent polynomial coefficient $f_{x}(t)$ of $m_{x}/t^{\ell
%(x)}$ having
%a nonzero term of nonnegative degree in $t.$ If none are found, then the
%algorithm is finished, and $^{J}C_{y}^{^{\prime}}=Fat$ $^{J}C_{y}^{^{\prime}}
%$. If one is found, we focus on an $x<y$ of maximal length with such a
%coeffecient.  Put $f(t)=f_{x}(t),$ and set $g(t)=f_{\geq0}(t)+f_{>0}(t^{-1})$.
%\ \ Reassign $Fat$ $^{J}C_{y}^{\prime}$, in terms of its old value\thinspace,
%as $Fat$ $^{J}C_{y}^{^{\prime}}-g(t)^{J}D_{s^{\prime}}^{\prime}%
%$, where $\mathbf{s}
%^{\prime}$ is a sequence of elements of $S$ whose product is reduced and equal
%to $x$ . \ Repeat these reassignments of $Fat$ $^{J}C_{y}^{\prime}$ until they
%can no longer be made, or, equivalenty,  $^{J}C_{y}^{^{\prime}}=Fat$
%$^{J}C_{y}^{^{\prime}}.$
%\end{algorithm}}}%
%BeginExpansion

\begin{algorithm}
For any given $y\in W^{J}$, we determine $^{J}C_{y}^{\prime}$ as an $R$-linear
combination of the basis elements $m_{x}$ of $M$:   Write $y=s_{1}s_{2}\cdots
s_{k}$ as a reduced product for a sequence $\mathbf{s}=(s_{1},s_{2}
,\ldots,s_{k})$ of elements of $S.$ Introduce a temporary variable $Fat$
$^{J}C_{y}^{\prime}$, initialized to $^{J}D_{\mathbf{s}}^{\prime}$ and
written as a linear combination of the elements $m_{x}/t^{\ell(x)}$, $x\in
W^{J}$.   Next, we  look in $Fat$ $^{J}C_{y}^{^{\prime}}$ for any $x<y$
with a\ Laurent polynomial coefficient $f_{x}(t)$ of $m_{x}/t^{\ell
(x)}$ having a nonzero term of non-negative degree in $t.$ If none are found, then the
algorithm is finished, and $^{J}C_{y}^{^{\prime}}=Fat$ $^{J}C_{y}^{^{\prime}}
$. If one is found, we focus on an $x<y$ of maximal length with such a
coefficient.  Put $f(t)=f_{x}(t),$ and set $g(t)=f_{\geq0}(t)+f_{>0}(t^{-1})$.
  Reassign $Fat$ $^{J}C_{y}^{\prime}$, in terms of its old value,
as $Fat$
$^{J}C_{y}^{^{\prime}}-g(t)^{J}D_{{\mathbf{s}}^{\prime}}^{\prime}$,
where $\mathbf{s}^{\prime}$ is a sequence of elements of $S$ whose
product is reduced and equal to $x$ .  Repeat these reassignments of
$Fat$ $^{J}C_{y}^{\prime}$ until they can no longer be made, or,
equivalently,  $^{J}C_{y}^{^{\prime}}=Fat$$^{J}C_{y}^{^{\prime}}.$
\end{algorithm}

\begin{proof}
We need to show that the algorithm terminates and gives the right answer. It
is fairly clear that the algorithm terminates, since the operations dealing
with a given $x<y$ only affect coefficients of $m_{z}/t^{\ell(z)}$ for
$z\leq x.$ Moreover, they result in $m_{x}/t^{\ell(x)}$ having a Laurent
polynomial coefficient $f_{x}(t)$ with no non-negative powers of $t$, a
coefficient that is undisturbed by later operations with elements in $W^{J}$
smaller than or unrelated to $x$ in the Bruhat--Chevalley order. (In fact, all
operations with $x<y$ of, say, maximal length with respect to having an
offending coefficient $f_{x}(t)$ for $m_{x}/t^{\ell(x)}$, can be done in
parallel.) Eventually, all $x<y$ in $W^{J}$ are exhausted, and the algorithm
terminates. At that point, all coefficients $f_{x}(t)$ for $m_{x}/t^{\ell(x)}$
in $Fat$ $^{J}C_{y}^{^{\prime}}$ have no non-negative powers of $\dot{t}\,$,
while the coefficient $f_{y}(t)=1$ from the initial $Fat$ $^{J}C_{y}^{\prime}$
has remained undisturbed. Thus, the above proposition implies we now have the
desired equation $Fat$ $^{J}C_{y}^{^{\prime}}=\ ^{J}C_{y}^{\prime}$.
\end{proof}

%TCIMACRO{\TeXButton{remark}{\newtheorem{remark}[proposition]{Remark}
%\begin{remark}
%We have here used many ingredients of \cite{De2}, and the algorithm we have
%obtained above may be viewed, philosophically, as a variation on the algorithm
%given in \cite[Algorithm 4.11, p.115]{De2}, sometimes called "Deodhar's
%algorithm".  Without going into too many details, our alternative uses the
%elements  $^{J}D_{s^{\prime}}$ in place of elements $^{J}C_{x}^{^{\prime}}$ in
%the reduction process, and the polynomials $g(t)=f_{\geq0}(t)+f_{>0}(t^{-1})$
%are used in place of the positive coefficient polynomials in $t+t^{-1}$
%guaranteed in the $^{J}C_{x}^{^{\prime}}$  case by \cite[Prop. 3.7, Cor.
%5.4]{De2}  The propostion above makes this work. There are, however, two
%advantages of our alternative:  First, unlike \cite[Algorithm 4.11,
%p.115]{De2} , the alternative algorithm does not require an $a$ $priori$
%positivity condition to gurantee its successful termination. Second, the
%alternative algorithm  has considerably less memory requirements when
%focussed on computing a single $^{J}C_{y}^{^{\prime}}$ , since the recursive
%calculation of elements $^{J}C_{x}^{^{\prime}}$ is avoided, together with any
%associated storage. The next section gives an illustration in a useful case.
%\end{remark}}}%
%BeginExpansion

\begin{remark}
We have here used many ingredients of \cite{De2}, and the algorithm
we have obtained above may be viewed, philosophically, as a
variation on the algorithm given in \cite[Algorithm 4.11,
p.~115]{De2}, sometimes called ``Deodhar's algorithm."  Without
going into too many details, our alternative uses the elements
$^{J}D_{s^{\prime}}$ in place of elements $^{J}C_{x}^{^{\prime}}$ in
the reduction process, and the polynomials
$g(t)=f_{\geq0}(t)+f_{>0}(t^{-1})$ are used in place of the positive
coefficient polynomials in $t+t^{-1}$ guaranteed in the
$^{J}C_{x}^{^{\prime}}$  case by \cite[Prop.~3.7, Cor.~5.4]{De2}.
The proposition above makes this work. There are, however, two
advantages of our alternative:  first, unlike \cite[Algorithm 4.11,
p.~115]{De2}, the alternative algorithm does not require an {\em a
priori} positivity condition to guarantee its successful
termination. Second, the alternative algorithm  has considerably
less memory requirements when focused on computing a single
$^{J}C_{y}^{^{\prime}}$, since the recursive calculation of elements
$^{J}C_{x}^{^{\prime}}$ is avoided, together with any associated
storage. The next section gives an illustration in a useful case.
\end{remark}

\section{An Example}

In this section we fix $W$ of affine type $\widetilde{A}_{n}$ with
$S=\{s_{0},$ $s_{1},\ldots,s_{n}\}.$ We suppose the indexing chosen as usual so
that products of successive elements, as well as $s_{n}s_{0}$, have order 3.
Fix $J=\{s_{1},\ldots,s_{n}\},$ so that $W_{J}$ is of type $A_{n}.$ As noted
above, we have the identification $P_{x,y}^{J}=P_{w_{J}^{0}x,w_{J}^{0}y}$ in
this case, for all $x,y\in W^{J}$.  Recall also that $^{J}\mu(x,y)$ denotes
the coefficient of $q^{(\ell(y)-\ell(x)-1)/2}$ in $P_{x,y}^{J}$, so that
$^{J}\mu(x,y)=\mu(w_{0}x,w_{0}y),$ where we have abbreviated $w_{J}
^{0}=w_{0}$.  Let $\varpi_{1},\ldots,\varpi_{n}$ be fundamental weights for a
root system of type $A_{n},$ and denote the integral weight lattice they
generate by $\Lambda$. Write elements $\sum_{i=1}^{n}
a_{i}\varpi_{i}$ of $\Lambda$ as $n$-tuples of integers $(a_{1},\ldots
,a_{n}).$ Let each $s_{i}$ with $0\,<i\leq n$ act on $\Lambda$ by reflection
in the $i^{th}$ fundamental root $\alpha_{i}$ (so that $s_{i}(\varpi
_{i})=\varpi_{i}-\alpha_{i}$ and $s_{i}(\varpi_{j})=\varpi_{j}$ for $j\neq
i$). Let $s_{0}$ act by reflection in the maximal root $\alpha_{0}$, followed
by translation via $-p\alpha_{0},$ where, for the moment, $p$ is just a fixed
positive integer. This gives an affine action of $W$ on $\Lambda,$ which we
next shift to give the standard ``dot" action: Put $\rho=\varpi_{1}%
+\cdots+\varpi_{n}$ and, for $\lambda\in\Lambda$ and $w\in W,$ define
$w\cdot\lambda=w(\lambda+\rho)-\rho$.  To emphasize the dependence of our
notation on $p$, we write $W\cong W_{p},$ viewing the left-hand side as an
abstract Coxeter group, and the right-hand side the group of affine
transformations giving its action on $\Lambda$ (with a recipe partly
involving translations by elements of $p\Lambda$). Assuming $p\geq n+1$, the
weights in $W\cdot-2\rho$ = $W_{p}\cdot-2\rho$ are in 1-1 correspondence with
the elements of $W.$ The dominant weights in $W_{p}\cdot-2\rho$ (those with
non-negative coefficients at each $\varpi_{i}$) are precisely those of the form
$w_{0}x\cdot-2\rho$ with $x\in W^{J}$. This fact is independent of
$p$, though, for fixed $x\in W^{J}$, the precise dominant weight
represented by $w_{0}x\cdot-2\rho$ will generally depend on $p.$ However, if
$w_{0}x\cdot-2\rho$ is $p$-restricted (has all coefficients $a_{i}$ of
fundamental roots in the range $0\leq a_{i}\leq p-1)$ for one choice of $p\geq
n+1,$ it can be shown to be $p$-restricted for any other such choice.

In \cite{ScXi} it is shown that, as $n$ grows, the values $\mu(w_{0}x,w_{0}y)$
for $x,y\in W^{J}$ get arbitrarily large, though they are bounded for fixed
$n.$ This is true even when the associated weights $w_{0}x\cdot-2\rho
,w_{0}y\cdot-2\rho$ are $p-$restricted. Left open was the important case where
$x=1$ was fixed and $n$ and $y$ were allowed to vary (keeping $w_{0}%
y\cdot-2\rho$ $\ p$-restricted). As discussed in \cite{Sc}, this case is
important because the values $\mu(w_{0},w_{0}y)$ give lower bounds on the
dimension of 1-cohomology groups with coefficients in the irreducible modules
$L(w_{0}y\cdot-2\rho)$ of the finite projective special linear groups
PSL$(n+1,q)\ $for $q$ a power of a sufficiently large prime, relevant to a
well-known conjecture of Guralnick.\footnote{In 1984, Guralnick conjectured
that there is a universal constant, call it $C$, such that dim$_{F}$H$^{1}(G,V)\leq C$
whenever $G$ is a finite group acting faithfully and
absolutely irreducibly as $F-$linear automorphisms of a vector space $V$ over a
field $F$ \cite{Gur}. The relevance of Kazhdan-Lusztig polynomials to this
conjecture was demonstrated in \cite{Sc}, showing\ dim$_{F}$ H$^{1}(G,V)$
$\geq\mu(w_{0},w_{0}y)$ for finite groups $G$ of Lie type acting on an
irreducible module\ $V=L(w_{0}y\cdot-2\rho)$. This gave for the first time
dimensions as large as $3.$ This remained the largest known value until the
2012 AIM conference, where computer calculations of Frank L\"ubeck led to
values of $\mu(w_{0},w_{0}y)$ in the hundreds, and to counterexamples to a
related 1961 conjecture of G.~E.~Wall on maximal subgroups. See  \cite{AIM}.
 The Guralnick conjecture, however, is still open, though current efforts
focus on understanding how  dim$_{F}$ H$^{1}(G,V)$ can grow with the rank
of an underlying root system for a finite group of Lie type,  rather than
trying to bound it universally. It is true, that, if the rank is fixed, then
there is a bound depending only on the rank, in either defining or cross
characteristic \cite{CPS09}, \cite{GurTie11}.} However, \cite{ScXi}
 does give a guess, when either $n$ is odd or divisible by $4,$ for a
$p-$restricted weight $w_{0}y\cdot-2\rho$ likely to give a large $\mu
(w_{0},w_{0}y)$. The guess may be described uniformly if we take $p=n+1,$ in
which case the guess reads (for all $n$ not congruent to $2$ modulo $4):$
\[
w_{0}y.-2\rho=(p-2)\rho-\alpha_{0}.
\]
For example, for $n=3,4,5,7,8$ these weights (in $p=n+1$ notation) are
$(2,1,2)$, $(2,3,3,2)$, $(3,4,4,4,3)$, $(4,5,5,5,5,4)$,
$(6,7,7,7,7,7,7,6)$. The corresponding values of $\mu(w_{0},w_{0}y)$ for the
first four had been previously computed, as 1, 2, 3, 469 as part of exhaustive
calculations\footnote{These calculations may be done by hand for $n=3,$ and the
remaining calculations by computer. For $n=4$ they were carried out by Chris
McDowell \cite[Prop.~3]{Sc}. The calculations for $n=7$ were done by Frank L\"ubeck and
confirmed independently by Tim Sprowl. Also, L\"ubeck did a similar exhaustive
calculation for $n=6$, determining a largest value of $16$ for $\mu
(w_{0},w_{0}y)$ for $p-$restricted $\mu(w_{0},w_{0}y),$ after earlier
calculations by Sprowl of values 4 and 5 for smaller weights. Some of these
calculations took place during the June 2012 AIM workshop, and the
remainder a few weeks later. See \cite{AIM}.} including all restricted
weights $w_{0}y\cdot-2\rho.$  We give here, using the algorithm of this
paper, the value of $\mu(w_{0},w_{0}y)$ for the $n=8$ weight $w_{0}%
y\cdot-2\rho=(6,7,7,7,7,7,7,6)$ as 36672. The full Kazhdan--Lusztig polynomial
$P_{w_{0},w_{0}y}$ is given below.\footnote{To be sure, the displayed equation is the result
of a 64-bit calculation, and can only be rigorously claimed to be correct modulo $2^{64}$. Known
theoretical bounds for the coefficients are not particularly good at this point, and even to
accurately pin down the coeffecient of $t\symbol{94}82$ below would require $11\times 64$ bit arithmetic, using
bounds based on \cite[Prop.~7.1]{ParkSt}. Fortunately, however, current interest is in a lower bound for this
coefficient, and all the coefficients are known to be positive.}

$$\begin{array}{l}
+ \ 36672t\symbol{94}82 + 329119t\symbol{94}80 + 1600603t\symbol{94}78 + 5782048t\symbol{94}76 + 17370114t\symbol{94}74 \\[2mm] + 45208788t\symbol{94}72 + 104312889t\symbol{94}70  + 216672871t\symbol{94}68 + 409222372t\symbol{94}66 + 707571983t\symbol{94}64 \\[2mm] + 1125993513t\symbol{94}62 + 1656221777t\symbol{94}60 + 2260164853t\symbol{94}58 + 2871480057t\symbol{94}56 + 3407386353t\symbol{94}54 \\[2mm] + 3787877798t\symbol{94}52 + 3955903667t\symbol{94}50 + 3891194815t\symbol{94}48 + 3613245907t\symbol{94}46 + 3173587791t\symbol{94}44 \\[2mm] + 2640964839t\symbol{94}42 + 2084968629t\symbol{94}40 + 1563002756t\symbol{94}38 + 1113178197t\symbol{94}36 + 753257475t\symbol{94}34 \\[2mm] + 484075798t\symbol{94}32 + 295159975t\symbol{94}30 + 170488857t\symbol{94}28 + 93076435t\symbol{94}26 + 47878089t\symbol{94}24 \\[2mm] + 23109923t\symbol{94}22 + 10411073t\symbol{94}20 + 4347162t\symbol{94}18 + 1667234t\symbol{94}16 + 580355t\symbol{94}14 + 180463t\symbol{94}12 \\[2mm] + 49052t\symbol{94}10 + 11300t\symbol{94}8 + 2107t\symbol{94}6 + 294t\symbol{94}4 + 26t\symbol{94}2 + 1t\symbol{94}0
\end{array}
$$

It would be difficult to make this calculation by using existing recursions
and exhaustively computing all Kazhdan--Lusztig basis elements\ $C_{w_{0}%
z}^{\prime}$ with $z\in W^{J}$ with $z\leq y:$ There are approximately $N=$
$1,700,000$ elements $z\in W^{J}$ with $\ell(z)\leq\ell(y)$ when\ $w_{0}%
y\cdot-2\rho=(6,7,7,7,7,7,7,6)$. Let us crudely estimate that, roughly half of
these elements satisfy $z\leq y,$ and that half the elements $x\in W^{J}$
satisfying $\ell(x)\leq\ell(z)$ also satisfy $x\leq z\,,$ at least when
$\ell(z)$ is modestly large. Comparison with linear orders now leads to a
guess that there are about ($N/2)^{2}/4=N^{2}/16$  such pairs. If we presume
the recursion would at least require knowing some information for every such
pair, recorded as a 32-bit pointer (say) to some small list (possibly evolving) of
further data, such as candidate Kazhdan--Lusztig polynomials, we are led to a
memory requirement of $N^{2}/4$ bytes, or about $\frac{2.89}{4}$ terabytes.
This is a disturbing estimate, to say the least.\footnote{To be sure, the only
information that ``really" needs to be stored is $\mu(w_{0}x,w_{0}z),$ and it
would not be hard to avoid any storage for those pairs $x\leq z$ where
$\mu(w_{0}x,w_{0}z)=0$.  However, even if this were to reduce storage
requirements to an acceptable level, existing recursions do not take the route
of such bare-bones storage. It would, of course, be an interesting project to
see if a new algorithm could be designed which did so, and ran in reasonable
time. The storage proposed is very close to the well-studied notion of a
$W$-graph defined by Kazhdan and Lusztig \cite{KL} It is much easier to
construct Kazhdan--Lusztig polynomials given the $W-$graph, than having to
extract the $\mu$ values from other Kazhdan--Lusztig polynomials as the
construction proceeds.}

By contrast, our original (32-bit) calculation required only 1.2 gigabytes of memory in a
fairly straightforward implementation. It's running time was quite acceptable,
given the task at hand, taking about 15 days using a relatively slow 2.2
gigahertz processor.\footnote{With an easy OpenMP parallelization and a single 8 cpu
computer of the same speed, this running time was cut down to 3.5 days, even with
64-bit arithmetic. This parallel version required
about 10 gigabytes, shared by the cpus, and ``confirmed"
the 32 bit results.} The calculation was carried out entirely by a C++
encoding of the algorithm above  by the second author, and posted on the
first author's webpage www.math.virginia.edu/\symbol{126}lls2l in January
2013. It was reported at the January 2013 AMS meeting, as well as subsequent
lectures in the first half of 2013 by the first author in Perth, Sydney, and
Zhangjiajie (ICRT6).

\end{document}